\documentclass[11pt, reqno]{amsart}
\usepackage{epsfig}
\usepackage{amsmath, amsthm, amssymb}

\usepackage{graphicx,xypic}
\numberwithin{equation}{section}
\newcommand{\heis}{{\bf H}}
\newcommand{\jacob}[1]{\ensuremath{\mathcal{J}(#1)}}
\newcommand{\RS}{\ensuremath{\Sigma_g}}
\newcommand{\mcg}[1]{\ensuremath{\mathcal{M}_{#1}}}
\newcommand{\spacetheta}{{\bf \Theta}_N^\Pi}

\newcommand{\qgr}{A}
\newcommand{\noproof}{\begin{flushright} \ensuremath{\square} \end{flushright}}
\DeclareMathOperator{\id}{id}
\DeclareMathOperator{\op}{Op}
\DeclareMathOperator{\Aut}{Aut}
\DeclareMathOperator{\Ann}{Ann}
\newtheorem{theorem}{Theorem}[section]

\newtheorem{proposition}[theorem]{Proposition}
\newtheorem{cor}[theorem]{Corollary}
\theoremstyle{definition}
\newtheorem{definition}[theorem]{Definition}

\theoremstyle{remark}
\newtheorem{remark}[theorem]{Remark}

\begin{document}

\title{Classical theta functions from a quantum group perspective}
\author{R{\u{a}}zvan Gelca}
\address{Department of Mathematics and Statistics, Texas Tech University, Lubbock, TX 79409-1042. USA.} \email{rgelca@gmail.com}
\thanks{Research of the first author partially supported by the NSF, award No. DMS 0604694}
\author{Alastair Hamilton}
\address{Department of Mathematics and Statistics, Texas Tech University, Lubbock, TX 79409-1042. USA.} \email{alastair.hamilton@ttu.edu}

\begin{abstract}
In this paper we construct the quantum group, at roots of unity, of abelian Chern-Simons theory.
We then use it to model classical theta functions and  the actions of the Heisenberg and modular groups on them.
\end{abstract}

\maketitle

\tableofcontents

\section{Introduction}\label{sec_introduction}

In 1989, Witten \cite{witten}
introduced a series of knot and 3-manifold invariants
based on a quantum field theory with the Chern-Simons lagrangian. Witten defined
these invariants starting with a compact simple Lie group,
the gauge group of the theory, and defining a path integral for each 3-manifold
as well as for  knots and links in such a 3-manifold. The case
most intensively studied is that where the gauge group is
$SU(2)$, which is related to the Jones polynomial of knots \cite{jones}.
Witten pointed out that the case where the gauge group is $U(1)$ (abelian
Chern-Simons theory) is
related to the linking number of knots and to classical theta functions.
Here and below we use the term ``classical'' not in the complex analytical
distinction between classical and canonical, but to specify the classical
theta functions as opposed to the non-abelian ones.

Abelian Chern-Simons theory was studied in \cite{andersen}, \cite{gelcauribe1},
\cite{manoliu1}, \cite{manoliu2} and \cite{moo}.
Our interest was renewed by the discovery, in \cite{gelcauribe1}, that
all constructs of abelian Chern-Simons theory can be derived from
the theory of theta functions. More precisely, in \cite{gelcauribe1}
it was shown that the action of the Heisenberg group on theta functions
discovered by A. Weil, and the action of the modular group lead naturally
to manifold invariants.

The current paper follows the same line. It has a two-fold goal. First,
we describe the quantum group of abelian Chern-Simons theory  which,
to our knowledge, has not been studied before.
The intuition for the construction comes from the properties of
theta functions.
Second, we interpret the classical theta functions, the action of the Heisenberg
group, and the action of the modular group in terms of vertex models using
this quantum group. We mention that the modular functor
of abelian Chern-Simons theory
will be described by the authors in a subsequent paper.

\subsection*{Notation and conventions}

If $V$ and $W$ are vector spaces, we denote the linear map that permutes $V$ and $W$ by
\[ P: V\otimes W \to W\otimes V, \qquad v\otimes w \mapsto  w\otimes v. \]

Given a Hilbert space $H$, we will denote the group of unitary transformations of $H$ by $U(H)$. In particular, we will denote the group of unitary transformations of $\mathbb{C}^n$ by $U(n)$.

Given a ring $R$, we denote the ring of formal power series over $R$ in a single variable $h$ by $R[[h]]$. We denote the ring of \emph{noncommutative} polynomials in the variables $X_1,\ldots,X_n$ by $R\langle X_1,\ldots,X_n\rangle$. The ring of integers modulo $N$ will be denoted by $\mathbb{Z}_N$.

Given a Riemann surface $\RS$, we will denote the mapping class group of $\RS$ by $\mcg{\RS}$. If $\mathfrak{g}$ is a Lie algebra, we will denote its universal enveloping algebra by $\mathcal{U}(\mathfrak{g})$.

\section{Classical theta functions from a topological perspective} \label{sec_thetafun}

In this section we  recall the essential facts from \cite{gelcauribe1}. Classical theta functions may be described as sections of a certain line bundle over the Jacobian variety $\jacob{\RS}$ associated to a closed genus $g$ Riemann surface $\RS$.

The Jacobian variety of a genus $g$ Riemann surface $\RS$ is constructed as follows (see \cite{farkaskra}). First choose a canonical basis of $H_1(\RS,\mathbb{Z})$ like the one in Figure \ref{fig_homology}. Such a basis is given by a collection of oriented simple closed curves
\[ a_1,a_2,\ldots, a_g, b_1,b_2,\ldots, b_g\]
which satisfy
\begin{displaymath}
\begin{split}
a_i\cdot a_j=b_i\cdot b_j &= 0 \\
a_i\cdot b_j &= \delta_{ij}
\end{split}
\end{displaymath}
with respect to the intersection form. These curves can also be interpreted as generators of the fundamental group, and in this respect they define a \emph{marking} on the surface in the sense of \cite{imayoshitaniguchi}. The complex structure on the surface $\RS$ together with this marking defines a point in the Teichm\"uller space $\mathcal{T}_g$, cf. \cite{imayoshitaniguchi}.

\begin{figure}[ht]
\centering
\resizebox{.50\textwidth}{!}{\includegraphics{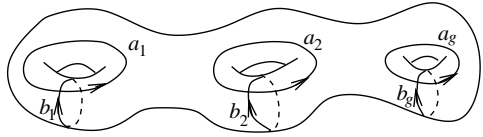}}
\caption{Any marked surface $\RS$ may be canonically identified (up to isotopy) with the surface in the figure above.}
\label{fig_homology}
\end{figure}

We recall from \cite{farkaskra} that given our marked Riemann surface $\RS$, there exists a unique set of holomorphic $1$-forms $\zeta_1,\zeta_2,\ldots, \zeta_g$ on $\RS$ such that
\[ \int _{a_j}\zeta_k dz=\delta_{jk}. \]
The matrix $\Pi\in M_g(\mathbb{C})$ whose entries are
\[ \pi_{jk}=\int_{b_j}\zeta_k dz \]
is symmetric with positive definite imaginary part.

\begin{definition}
The Jacobian variety $\jacob{\RS}$ of the marked Riemann surface $\RS$ is the quotient
\[ \jacob{\RS}:= \mathbb{C}^g/L\]
of $\mathbb{C}^g$ by the lattice subgroup
\[ L:=\{ t_1\lambda_1+\cdots+t_{2g}\lambda_{2g}: \ t_1,\ldots,t_{2g}\in\mathbb{Z}\} \]
that is spanned by the columns $\lambda_i$ of the $g \times 2g$ matrix
\[ \lambda:=(I_g,\Pi), \]
which we call the \emph{period matrix}.
\end{definition}

\begin{remark}
This identifies $\jacob{\RS}$ as an abelian variety whose complex structure depends upon the matrix $\Pi$. Changing the marking on $\RS$ whilst leaving the complex structure of $\RS$ fixed gives rise to a biholomorphic Jacobian variety, albeit with a different period matrix.
\end{remark}

The period matrix $\lambda$ gives rise to an invertible $\mathbb{R}$-linear map
\begin{displaymath}
\begin{array}{rcl}
\mathbb{R}^{2g}=\mathbb{R}^g\times\mathbb{R}^g & \to & \mathbb{C}^g \\
(x,y) & \mapsto & \lambda (x, y)^T = x + \Pi y
\end{array}
\end{displaymath}
which descends to a diffeomorphism
\begin{equation} \label{eqn_realcoord}
\mathbb{R}^{2g}/\mathbb{Z}^{2g} \approx \mathbb{C}^g/L=\jacob{\RS}.
\end{equation}
Using this system of real coordinates on $\jacob{\RS}$, we may canonically and unambiguously define a symplectic form
\begin{eqnarray*}
\omega=\sum_{j=1}^g dx_j\wedge dy_j
\end{eqnarray*}
on $\jacob{\RS}$. This symplectic form allows us to identify $\jacob{\RS}$ with the phase space of a classical mechanical system, and to address questions about quantization.

Choose Planck's constant to be $h=\frac{1}{N}$, where $N$ is a positive \emph{even} integer. The Hilbert space of our quantum theory is obtained by applying the procedure of geometric quantization \cite{sniatycki}, \cite{woodhouse}. It may be constructed as the space of holomorphic sections of a certain holomorphic line bundle over the Jacobian variety $\jacob{\RS}$ which is obtained as the tensor product of a line bundle with curvature $-\frac{2\pi i}{h}\omega=-2\pi i N\omega$ and a half-density. Such sections may be identified, in an obvious manner, with holomorphic functions on $\mathbb{C}^g$ satisfying certain periodicity conditions determined by the period matrix $\lambda$, leading to the following definition.

\begin{definition}
The space of classical theta functions $\spacetheta(\RS)$ of the marked Riemann surface $\RS$ is the vector space consisting of all holomorphic functions $f:\mathbb{C}^g\to\mathbb{C}$ satisfying the periodicity conditions
\begin{displaymath}
\begin{split}
f(z+\lambda_j) &= f(z), \\
f(z+\lambda_{g+j}) &= e^{-2\pi iNz_j-\pi i N\pi_{jj}}f(z);
\end{split}
\end{displaymath}
for $j=1,\ldots,g$; where $\lambda_i$ denotes the $i$th column of the period matrix $\lambda$. The space of classical theta functions may be given the structure of a Hilbert space by endowing it with the inner product
\begin{equation} \label{eqn_hilbertform}
\langle f,g \rangle := (2N)^{\frac{g}{2}} (\det\Pi_{\mathrm{I}})^{\frac{1}{2}} \int_{[0,1]^{2g}} f(x,y)\overline{g(x,y)} e^{-2\pi N y^T \Pi_{\mathrm{I}} y} dx dy,
\end{equation}
where $\Pi_{\mathrm{I}}\in M_g(\mathbb{R})$ denotes the imaginary part of $\Pi$.
\end{definition}

The space of theta functions depends only on the complex structure of $\RS$ and not on the marking, in the sense that choosing a different marking on $\RS$ yields an isomorphic Hilbert space. However, the marking on $\RS$ specifies a particular orthonormal basis consisting of the theta series
\[ \theta_\mu^{\Pi}(z):=\sum_{n\in {\mathbb Z}^g}e^{2\pi i N[\frac{1}{2}\left(\frac{\mu}{N}+n\right)^T \Pi \left(\frac{\mu}{N}+n\right)+\left(\frac{\mu}{N}+n\right)^Tz]}, \quad \mu\in \mathbb{Z}_N^g.\]
Hence each point in $\mathcal{T}_g$ gives rise to a space of theta functions endowed with a preferred basis.

Given $p,q\in \mathbb{Z}^{g}$, consider the exponential function
\begin{displaymath}
\begin{array}{ccc}
\jacob{\RS} & \to & \mathbb{C} \\
(x,y) & \mapsto & e^{2\pi i (p^T x + q^T y)}
\end{array}
\end{displaymath}
defined on $\jacob{\RS}$ in the coordinates \eqref{eqn_realcoord}. Applying the Weyl quantization procedure in the momentum representation \cite{andersen}, \cite{gelcauribe1} to such a function, one obtains the operator
\begin{equation} \label{eqn_operator}
O_{pq}:= \op\left(e^{2\pi i (p^T x + q^T y)}\right):\spacetheta(\RS) \to \spacetheta(\RS)
\end{equation}
that acts on the theta series by
\[ O_{pq}[\theta_\mu^\Pi] = e^{-\frac{\pi i}{N}p^Tq-\frac{2\pi i}{N}\mu^Tq}\theta_{\mu+p}^\Pi. \]

\begin{definition}
Given a nonnegative integer $g$, the \emph{Heisenberg group} $\heis(\mathbb{Z}^g)$ is the group
\[ \heis(\mathbb{Z}^g):=\{(p,q,k): \; p,q\in\mathbb{Z}^g, k\in \mathbb{Z}\} \]
with underlying multiplication
\begin{eqnarray*}
(p,q,k)(p',q',k')=\left(p+p',q+q',k+k'+\sum_{j=1}^g\left|
\begin{array}{rr}
p_j& q_j\\
p'_j&q'_j
\end{array}
\right|
\right)
\end{eqnarray*}

Given an even integer $N$, the \emph{finite Heisenberg group} $\heis(\mathbb{Z}_N^g)$ is the quotient of the Heisenberg group $\heis(\mathbb{Z}^g)$ by the normal subgroup consisting of all elements of the form
\[ (p,q,2k)^N=(Np,Nq,2Nk); \quad p,q\in\mathbb{Z}^g, k\in\mathbb{Z}. \]
\end{definition}

\begin{remark} \label{rem_extform}
The finite Heisenberg group is a $\mathbb{Z}_{2N}$-extension of $\mathbb{Z}_N^{2g}$, and consequently has order $2N^{2g+1}$. One should point out that the group $\heis(\mathbb{Z}^g)$ can be interpreted as the $\mathbb{Z}$-extension of
\begin{equation} \label{eqn_canbasis}
\begin{array}{ccc}
H_1(\RS,\mathbb{Z}) & = & \mathbb{Z}^g \times \mathbb{Z}^g \\
\sum_{i=1}^g (p_i a_i + q_i b_i) & \leftrightharpoons & (p,q)
\end{array}
\end{equation}
by the cocycle defined by the intersection form.
\end{remark}

The operators \eqref{eqn_operator} generate a subgroup of the group $U\left(\spacetheta(\RS)\right)$ of unitary operators on $\spacetheta(\RS)$, which may be identified with the finite Heisenberg group as follows.

\begin{proposition}
Given a marked Riemann surface $\RS$ of genus $g$ and a positive even integer $N$, the subgroup $G$ of the group of unitary operators on $\spacetheta(\RS)$ that is generated by all operators of the form
\[ O_{pq}=\op\left(e^{2\pi i (p^T x + q^T y)}\right); \quad p,q \in \mathbb{Z}^g \]
is isomorphic to the finite Heisenberg group $\heis(\mathbb{Z}_N^g)$:
\begin{equation} \label{eqn_schrorep}
\begin{array}{ccc}
\heis(\mathbb{Z}_N^g) & \cong & G \\
(p,q,k) & \mapsto & e^{\frac{k\pi i}{N}} O_{pq}
\end{array}
\end{equation}
\end{proposition}

\begin{proof}
The proof of this proposition is a simple check, see Proposition 2.3 of \cite{gelcauribe1}.
\end{proof}

The representation of $\heis(\mathbb{Z}_N^g)$ on the space of theta functions $\spacetheta(\RS)$ defined by \eqref{eqn_schrorep}, which we refer to as the \emph{Schr\"odinger representation}, was first discovered by A. Weil \cite{weil} by examining translations in the line bundle over the Jacobian. Like the Schr\"odinger representation of the Heisenberg group with real entries, it  satisfies a Stone-von Neumann theorem.

\begin{theorem} \label{thm_stone}
Any irreducible unitary representation of the finite Heisenberg group in which the element $(0,0,1) \in \heis(\mathbb{Z}_N^g)$ acts as multiplication by the scalar $e^{\frac{\pi i}{N}}$ is unitarily equivalent to the Schr\"odinger representation \eqref{eqn_schrorep}.
\end{theorem}

\begin{proof}
The proof is standard, see Theorem 2.4 of \cite{gelcauribe1}.
\end{proof}

This Stone-von Neumann theorem provides a reason for the existence of the action of the mapping class group on theta functions, whose discovery can be traced back to the nineteenth century. An element $h$ of the mapping class group $\mcg{\RS}$ of $\RS$ induces a linear endomorphism $h_*$ of $H_1(\RS,\mathbb{Z})$ preserving the intersection form. By Remark \ref{rem_extform}, this endomorphism gives rise to an automorphism
\begin{equation} \label{eqn_mcginfheisact}
\begin{array}{ccc}
\heis(\mathbb{Z}^g) & \to & \heis(\mathbb{Z}^g) \\
((p,q),k) & \mapsto & (h_*(p,q),k)
\end{array}
\end{equation}
of the Heisenberg group. This automorphism descends to an automorphism of the finite Heisenberg group $\heis(\mathbb{Z}_N^g)$, yielding an action of the mapping class group of $\RS$ on $\heis(\mathbb{Z}_N^g)$;
\begin{equation} \label{eqn_mcgheisact}
\begin{array}{ccc}
\mcg{\RS} & \to & \Aut\left(\heis(\mathbb{Z}_N^g)\right) \\
h & \mapsto & \tilde{h}:((p,q),k) \mapsto (h_*(p,q),k)
\end{array}
\end{equation}
Pulling the Schr\"odinger representation \eqref{eqn_schrorep} back via the automorphism $\tilde{h}$ yields another representation
\[ (p,q,k) \mapsto e^{\frac{k\pi i}{N}}O_{h_*(p,q)} \]
of the finite Heisenberg group $\heis(\mathbb{Z}_N^g)$. By Theorem \ref{thm_stone}, this representation is unitarily equivalent to the Schr\"odinger representation, and therefore there exists a unitary map
\[ \rho(h):\spacetheta(\RS)\to\spacetheta(\RS), \]
satisfying the \emph{exact Egorov identity}
\[ e^{\frac{k\pi i}{N}}O_{h_*(p,q)}=\rho(h)\circ\left(e^{\frac{k\pi i}{N}} O_{pq}\right)\circ\rho(h)^{-1}; \quad p,q\in\mathbb{Z}^g, k\in\mathbb{Z}. \]
By Schur's Lemma, $\rho(h)$ is well-defined up to multiplication by a complex scalar of unit modulus. Consequently, this construction yields a projective unitary representation
\begin{displaymath}
\begin{array}{ccc}
\mcg{\RS} & \to & U\left(\spacetheta(\RS)\right)/U(1) \\
h & \mapsto & \rho(h)
\end{array}
\end{displaymath}
of the mapping class group on the space of theta functions known as the \emph{Hermite-Jacobi action}. The maps $\rho(h)$ can be described as \emph{discrete Fourier transforms} as we will explain below.

The Schr\"odinger representation can be obtained as an induced representation, and this allows us to relate theta functions to knots without the use of Witten's quantum field theoretic insights.

Consider the submodule
\[ \mathbf{L}:=\{(0,q):q\in\mathbb{Z}^g\} \]
of $\mathbb{Z}^g\times\mathbb{Z}^g$ that is isotropic with respect to the intersection pairing induced by \eqref{eqn_canbasis} and let
\[ \tilde{\mathbf{L}}:=\{(p,q,k)\in\heis(\mathbb{Z}^g):(p,q)\in\mathbf{L}\} \]
be the corresponding maximal abelian subgroup of $\heis(\mathbb{Z}^g)$. Denote the image of $\tilde{\mathbf{L}}$ in $\heis(\mathbb{Z}_N^g)$ under the natural projection by $\tilde{\mathbf{L}}_N$. Being abelian, this group has only one-dimensional irreducible representations, which are therefore characters. In view of the Stone-von Neumann theorem, one chooses the character
\[ \chi_{\mathbf{L}}:\tilde{\mathbf{L}}_N \to \mathbb{C} \]
given by $\chi_{\mathbf{L}}(p,q,k):=e^{\frac{k\pi i}{N}}$.

Consider the group algebras $\mathbb{C}[\heis(\mathbb{Z}_N^g)]$ and $\mathbb{C}[\tilde{\mathbf{L}}_N]$. Note that the latter acts on $\mathbb{C}$ by the character $\chi_{\mathbf{L}}$. The induced representation is
\begin{eqnarray*}
\mbox{Ind}_{\tilde{{\mathbf{L}}}_N}^{\heis(\mathbb{Z}_N^g)}=\mathbb{C}[\heis(\mathbb{Z}_N^g)] \otimes_{\mathbb{C}[\tilde{\mathbf{L}}_N]} \mathbb{C},
\end{eqnarray*}
with $\heis(\mathbb{Z}_N^g)$ acting on the left in the first factor of the tensor product.

The vector space of the representation can be described as the complex vector space, which we denote by $\mathcal{H}_{N,g}(\mathbf{L})$, that is obtained as the quotient of $\mathbb{C}[\heis(\mathbb{Z}_N^g)]$ by the relations
\begin{equation} \label{eqn_indreprel}
\chi_{\mathbf{L}}(u_1)u = uu_1; \quad u\in\heis(\mathbb{Z}_N^g), u_1\in\tilde{\mathbf{L}}_N.
\end{equation}
Denote the quotient map by
\[ \pi_{\mathbf{L}}:\mathbb{C}[\heis(\mathbb{Z}_N^g)] \to \mathcal{H}_{N,g}(\mathbf{L}). \]
The left regular action of $\heis(\mathbb{Z}_N^g)$ on $\mathbb{C}[\heis(\mathbb{Z}_N^g)]$ descends to an action of $\heis(\mathbb{Z}_N^g)$ on the quotient $\mathcal{H}_{N,g}(\mathbf{L})$ that gives rise to the induced representation.

\begin{proposition} \label{prop_thetalinks}
The map
\begin{equation} \label{eqn_thetalinksiso}
\begin{array}{ccc}
\spacetheta(\RS) & \to & \mathcal{H}_{N,g}(\mathbf{L}) \\
\theta_\mu^{\Pi} & \mapsto & \pi_{\mathbf{L}}[(\mu,0,0)]
\end{array}
\end{equation}
defines an $\heis(\mathbb{Z}_N^g)$-equivariant $\mathbb{C}$-linear isomorphism between the space of theta functions $\spacetheta(\RS)$ equipped with the Schr\"odinger representation and the space $\mathcal{H}_{N,g}(\mathbf{L})$ equipped with the left regular action of the finite Heisenberg group.
\end{proposition}

\begin{proof}
One may easily check that the map
\begin{displaymath}
\begin{array}{ccc}
\mathcal{H}_{N,g}(\mathbf{L}) & \to & \spacetheta(\RS) \\
\pi_{\mathbf{L}}[(p,q,k)] & \mapsto & e^{\frac{(k-p^T q)\pi i}{N}}\theta_{p}^\Pi
\end{array}
\end{displaymath}
describes a well-defined inverse to \eqref{eqn_thetalinksiso}.
\end{proof}

\begin{remark}
Let $h$ be an element of the mapping class group of $\Sigma_g$ and set $\mathbf{L}':=h_*(\mathbf{L})$. As above, we may construct the vector space $\mathcal{H}_{N,g}(\mathbf{L}')$ as the quotient of $\mathbb{C}[\heis(\mathbb{Z}_N^g)]$ by the same relations \eqref{eqn_indreprel}, where $\mathbf{L}$ is replaced by $\mathbf{L}'$ and the formula for the character $\chi_{\mathbf{L}'}$ is the same as that for $\chi_{\mathbf{L}}$.

Consider the map $\mathcal{H}_{N,g}(\mathbf{L})\to \mathcal{H}_{N,g}(\mathbf{L}')$ given by
\begin{eqnarray}\label{eqn_fourier}
\pi_{\mathbf{L}}(\mathbf{u}) \mapsto \frac{1}{\left[\tilde{\mathbf{L}}_N : (\tilde{\mathbf{L}}_N \cap \tilde{\mathbf{L}}'_N)\right]^{\frac{1}{2}}} \sum_{u_1\in \tilde{\mathbf{L}}_N/(\tilde{\mathbf{L}}_N\cap \tilde{\mathbf{L}}'_N)} \chi_{\mathbf{L}}(u_1)^{-1}\pi_{\mathbf{L}'}(\mathbf{u}u_1).
\end{eqnarray}
where $\mathbf{u}\in\mathbb{C}[\heis(\mathbb{Z}_N^g)]$. On the other hand, the automorphism $\tilde{h}$ of $\heis(\mathbb{Z}_N^g)$ defined by \eqref{eqn_mcgheisact} induces a canonical identification
\begin{displaymath}
\begin{array}{ccc}
\mathcal{H}_{N,g}(\mathbf{L}) & \cong & \mathcal{H}_{N,g}(\mathbf{L}') \\
\pi_{\mathbf{L}}(\mathbf{u}) & \mapsto & \pi_{\mathbf{L}'}(\tilde{h}(\mathbf{u}))
\end{array}
\end{displaymath}
Composing \eqref{eqn_fourier} with the inverse of this map yields an endomorphism of $\mathcal{H}_{N,g}(\mathbf{L})$ and consequently, by Proposition \ref{prop_thetalinks}, an endomorphism of $\spacetheta(\RS)$. This endomorphism
is (a unitary representative for) $\rho(h)^{-1}$. It is formula \eqref{eqn_fourier} that identifies $\rho(h)$ as a discrete Fourier transform.
\end{remark}

Denote by $L(\spacetheta(\RS))$ the algebra of $\mathbb{C}$-linear endomorphisms of $\spacetheta(\RS)$. The Schr\"odinger representation \eqref{eqn_schrorep} provides a way to describe this space of linear operators in terms of the finite Heisenberg group $\heis(\mathbb{Z}_N^g)$.

\begin{proposition} \label{prop_thetaops}(Proposition 2.5 in \cite{gelcauribe1})
The quotient of the algebra $\mathbb{C}[\heis(\mathbb{Z}_N^g)]$ by the ideal $I$ generated by the single relation
\[ (0,0,1)=e^{\frac{\pi i}{N}}(0,0,0) \]
is isomorphic, via the Schr\"odinger representation, to the algebra of linear operators $L(\spacetheta(\RS))$:
\begin{displaymath}
\begin{array}{ccc}
\mathbb{C}[\heis(\mathbb{Z}_N^g)]/I & \cong & L(\spacetheta(\RS)) \\
(p,q,k) & \mapsto & e^{\frac{k\pi i}{N}} O_{pq}
\end{array}
\end{displaymath}
\end{proposition}

Thus by Proposition \ref{prop_thetalinks}, the map \eqref{eqn_thetalinksiso} intertwines the Schr\"odinger representation and the left action of the finite Heisenberg group. As explained in \cite{gelcauribe1}, this abstract version of the Schr\"odinger representation has topological flavor, which allows us to model the space of theta functions, the Schr\"odinger representation, and the Hermite-Jacobi action using the skein modules of the linking number introduced by Przytycki in \cite{przytycki2}. Here is how this is done.

For an oriented three-manifold $M$, consider the free $\mathbb{C}[t,t^{-1}]$-module whose basis is the set of isotopy classes of framed oriented links that are contained in the interior of $M$, including the empty link $\emptyset$. Factor this by the $\mathbb{C}[t,t^{-1}]$-submodule spanned by the elements from Figure \ref{fig_skeinrelations}, where the two terms depict framed links that are identical, except in an embedded ball, in which they look as shown and have the blackboard framing; the orientation in Figure \ref{fig_skeinrelations} that is induced by the orientation of $M$ must coincide with the canonical orientation of $\mathbb{R}^3$. In other words, we are allowed to smooth each crossing, provided that we multiply with the appropriate power of $t$, and we are allowed to delete trivial link components. The result of this factorization is called the \emph{linking number skein module} of $M$ and is denoted by $\mathcal{L}_t(M)$. Its elements are called skeins. The name is due to the fact that the skein relations (Figure \ref{fig_skeinrelations}) are used in computing the Gaussian linking number of two curves.

\begin{figure}[ht]
\centering
\resizebox{.50\textwidth}{!}{\includegraphics{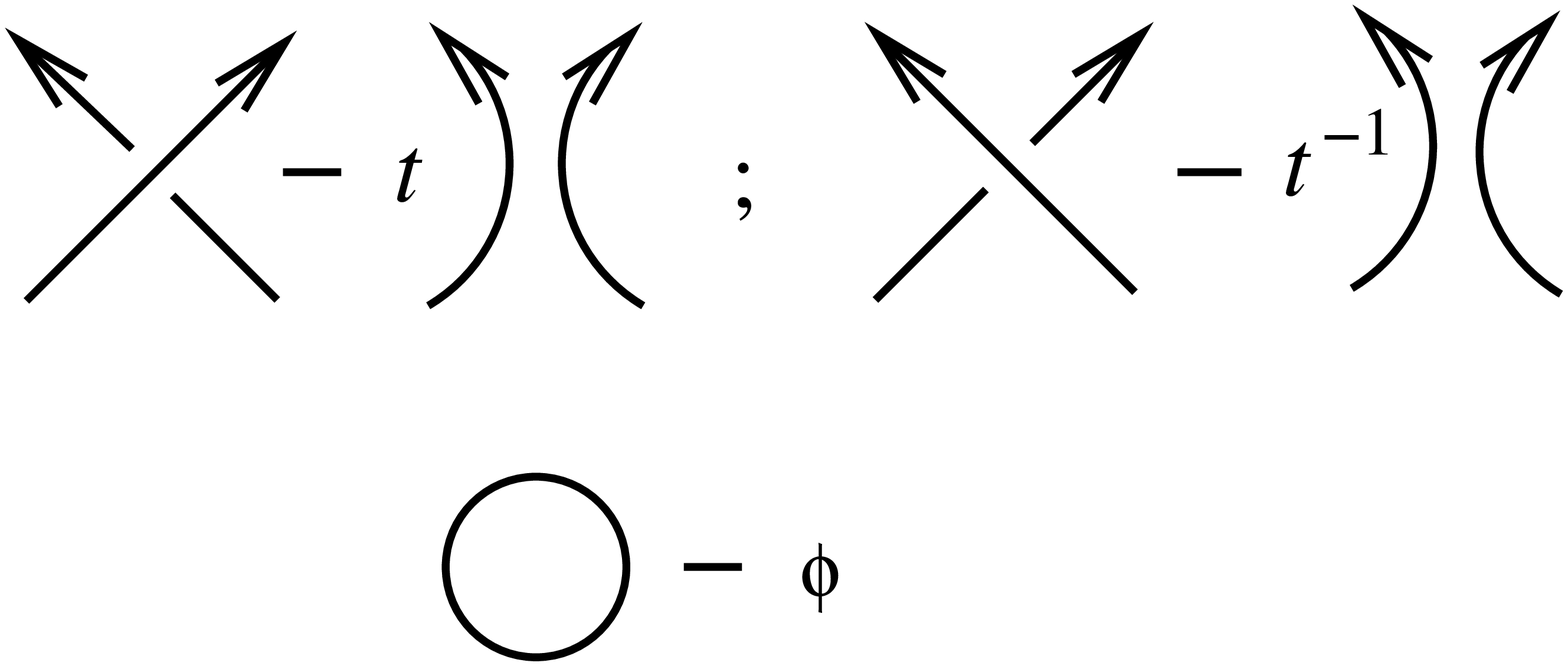}}
\caption{The skein relations defining $\mathcal{L}_t(M)$.}
\label{fig_skeinrelations}
\end{figure}

From the even positive integer $N$, we define the \emph{reduced linking number skein module} of $M$, denoted by $\widetilde{\mathcal{L}}_t(M)$, to be the $\mathbb{C}[t,t^{-1}]$-module obtained as the quotient of $\mathcal{L}_t(M)$ by the relations
\[ t\cdot L=e^{\frac{\pi i}{N}}\cdot L \quad\text{and}\quad \gamma^{\| N}\cup L=L \]
where the first relation holds for all links $L$, and the second relation holds for all oriented framed simple closed curves $\gamma$ and links $L$ disjoint from $\gamma$, and $\gamma^{\| N}$ is the multicurve\footnote{We use this notation to distinguish this from an $N$-fold product, which will appear later.} in a regular neighborhood of $\gamma$ disjoint from $L$ obtained by replacing $\gamma$ by $N$ parallel copies of it (where ``parallel'' is defined using the framing of $\gamma$).

If $M=\RS\times [0,1]$, then the identification
\begin{equation} \label{eqn_surfaceglue}
(\RS\times [0,1])\bigsqcup_{\begin{subarray}{c} \RS\times\{0\} \\ =\RS\times\{1\} \end{subarray}}(\RS\times [0,1]) \approx \RS\times [0,1]
\end{equation}
induces a multiplication of skeins in $\mathcal{L}_t(\RS\times [0,1])$ and $\widetilde{\mathcal{L}}_t(\RS\times [0,1])$ which transforms them into algebras.

Moreover, the identification
\begin{equation} \label{eqn_bdryglue}
(\partial M\times [0,1])\bigsqcup_{\begin{subarray}{c} \partial M\times\{0\} \\ =\partial M \end{subarray}} M\approx M,
\end{equation}
canonically defined up to isotopy, makes $\mathcal{L}_t(M)$ into a $\mathcal{L}_t(\partial M\times [0,1])$-module and $\widetilde{\mathcal{L}}_t(M)$ into a $\widetilde{\mathcal{L}}_t(\partial M\times [0,1])$-module.

We are only interested in the particular situation where the cylinder $\RS\times [0,1]$ is glued to the boundary of the standard handlebody $H_g$ of genus $g$, which we take to be the $3$-manifold that is enclosed by the surface depicted in Figure \ref{fig_homology}. The marking on $\RS$ leads to a canonical (up to isotopy) identification
\[ \partial H_g = \RS \]
making $\mathcal{L}_t(H_g)$ into a $\mathcal{L}_t(\RS\times [0,1])$-module and $\widetilde{\mathcal{L}}_t(H_g)$ into a $\widetilde{\mathcal{L}}_t(\RS\times [0,1])$-module. Note that any link in a cylinder over a surface is skein-equivalent to a link in which every link component is endowed with a framing that is parallel to the surface. Consider the curves
\[ a_1,a_2,\ldots, a_g, b_1,b_2,\ldots, b_g\]
from Figure \ref{fig_homology} that define the marking on $\RS$, equipped with this parallel framing.

Here is a rephrasing of Theorem 4.5 and Theorem 4.7 in \cite{gelcauribe1}.
\begin{theorem} \label{thm_linkingheisenberg}
The algebras $\mathbb{C}[\heis(\mathbb{Z}^g)]$ and $\mathcal{L}_t(\RS\times [0,1])$ are isomorphic, with the isomorphism defined by the map
\begin{equation} \label{eqn_linkheis}
\begin{array}{ccc}
\mathbb{C}[\heis(\mathbb{Z}^g)] & \cong & \mathcal{L}_t(\RS\times [0,1]) \\
(p,q,k) & \mapsto & t^{(k-p^T q)} a_1^{p_1}\cdots a_g^{p_g}b_1^{q_1}\cdots b_g^{q_g}
\end{array}
\end{equation}
This isomorphism is equivariant with respect the action of the mapping class group of $\RS$, where $\mcg{\RS}$ acts on the left by \eqref{eqn_mcginfheisact} and on the right in an obvious fashion.

Furthermore, this isomorphism descends, using Proposition \ref{prop_thetaops}, to an isomorphism of the algebras
\begin{equation} \label{eqn_linkthetarep}
L(\spacetheta(\RS)) \cong \mathbb{C}[\heis(\mathbb{Z}_N^g)]/\left\langle (0,0,1)=e^{\frac{\pi i}{N}}(0,0,0) \right\rangle \cong \widetilde{\mathcal{L}}_t(\RS\times [0,1]).
\end{equation}
\end{theorem}

Consider the framed curves
\[ a_1,a_2, \ldots ,a_g \]
on $\RS=\partial H_g$. These curves give rise to framed curves in the handlebody $H_g$ which, by an abuse of notation, we denote in the same way. By \eqref{eqn_linkthetarep}, $\widetilde{\mathcal{L}}_t(H_g)$ is a $L(\spacetheta(\RS))$-module.

\begin{theorem} \label{thm_thetaskein}(Theorem 4.7 in \cite{gelcauribe1}) 
The map
\begin{equation} \label{eqn_thetaskeinmap}
\begin{array}{ccc}
\spacetheta(\RS) & \to & \widetilde{\mathcal{L}}_t(H_g) \\
\theta_\mu^\Pi & \mapsto & a_1^{\mu_1}\cdots a_g^{\mu_g}
\end{array}
\end{equation}
is an isomorphism of $L(\spacetheta(\RS))$-modules.
\end{theorem}

\begin{proof}
See Theorem 4.7 of \cite{gelcauribe1}.
\end{proof}

Consider the orientation reversing diffeomorphism
\begin{equation} \label{eqn_heegaardmap}
\begin{array}{ccc}
\partial H_g & \approx & \partial H_g \\
a_i & \mapsto & b_i \\
b_i & \mapsto & a_i
\end{array}
\end{equation}
of the boundary of the standard handlebody $H_g$ that is canonically determined (up to isotopy) by the above action on the marking depicted in Figure \ref{fig_homology}. This identification gives rise to a Heegaard splitting
\[ H_g \bigsqcup_{\partial H_g \approx \partial H_g} H_g = S^3 \]
of $S^3$ in which the leftmost handlebody corresponds to the interior of Figure \ref{fig_homology} and the rightmost handlebody corresponds to the exterior. In turn, this Heegaard splitting defines a pairing
\begin{equation}\label{eqn_pairingskein}
[\cdot,\cdot ]:\spacetheta(\RS)\otimes \spacetheta(\RS)\cong \widetilde{\mathcal{L}}_t(H_g)\otimes\widetilde{\mathcal{L}}_t(H_g) \to \widetilde{\mathcal{L}}_t(S^3)=\mathbb{C},
\end{equation}
where we make use of Theorem \ref{thm_thetaskein} and the fact\footnote{Note that in $S^3$, the relation $\gamma^{\| N}\cup L = L$ is redundant, as using the skein relations of Figure \ref{fig_skeinrelations}, any $N$-fold multicurve may be disentangled from any link and transformed into a union of unknots, which may then be deleted.} that $\widetilde{\mathcal{L}}_t(S^3)$ is generated by the unknot. One may compute directly that
\begin{eqnarray*}
[\theta_\mu^\Pi,\theta_{\mu'}^\Pi]=t^{-2\mu^T\mu'}; \quad \mu,\mu'\in\mathbb{Z}_N^g.
\end{eqnarray*}
Since $t=e^{\frac{\pi i}{N}}$ is a primitive $2N$th root of unity, it follows from this formula that this pairing is non-degenerate (the inverse matrix is $\frac{1}{N^g}t^{2\mu^T\mu'}$). It is important to point out that this pairing is not the inner product \eqref{eqn_hilbertform}.

Let us now turn to the Hermite-Jacobi action. Recall that any element $h$ of the mapping class group of $\RS$ can be represented as surgery on a framed link in $\RS\times [0,1]$. Suppose that $L$ is a framed link in $\RS\times [0,1]$ such that the $3$-manifold $K$ that is obtained from $\RS\times [0,1]$ by surgery along $L$ is diffeomorphic to $\RS\times [0,1]$ by a diffeomorphism (where $K$ is the domain and $\RS\times [0,1]$ is the codomain) that is the identity on $\RS\times \{1\}$ and $h\in\mcg{\RS}$ on $\RS\times \{0\}$. Of course, not all links have this property, but for those that do, the diffeomorphism $h_L:=h$ is well-defined up to isotopy.

In particular, if $T$ is a simple closed curve on $\RS$, then a Dehn twist along $T$ is represented by the curve $T^+$ that is obtained from $T$ by endowing $T$ with the framing that is parallel to the surface and placing a single positive twist in $T$ (here, our convention for a Dehn Twist is such that a Dehn twist along the curve $b_1$ in Figure \ref{fig_homology} maps $a_1\in H_1(\RS)$ to $a_1+b_1$). The inverse twist is represented by the curve $T^-$ obtained by placing a negative twist in $T$.

By Theorem \ref{thm_linkingheisenberg}, any linear operator on $\spacetheta(\RS)$ may be uniquely represented by a skein in $\widetilde{\mathcal{L}}_t(\RS\times [0,1])$. To describe the skein associated to the discrete Fourier transform $\rho(h)$, we need the following definition.

\begin{definition}
Denote by $\Omega$ the skein in the solid torus depicted in Figure \ref{fig_omega} multiplied by $N^{-\frac{1}{2}}$.
\begin{figure}[ht]
\centering
\scalebox{.30}{\includegraphics{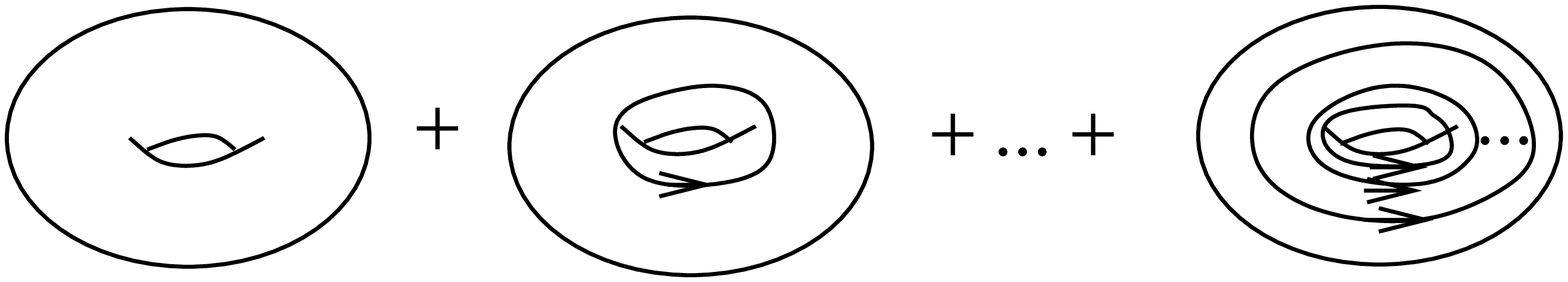}}
\caption{The skein $\Omega$. The sum has $N$ terms.}
\label{fig_omega}
\end{figure}

For a framed link $L$ in a 3-manifold $M$, denote by $\Omega(L)\in\mathcal{L}_t(M)$ the skein that is obtained by replacing each component of $L$ by $\Omega$ using the framing on $L$. More specifically, if $L$ is the disjoint union
\[ L=L_1\cup L_2\cup\cdots\cup L_m \]
of closed framed curves $L_1,\ldots,L_m$, then
\[ \Omega(L):=\frac{1}{N^{\frac{m}{2}}}\sum_{i_1,\ldots,i_m=0}^{N-1} L_1^{\| i_1}\cup L_2^{\| i_2}\cup\cdots\cup L_m^{\| i_m} \]
Note that $\Omega(L)$ is independent of the orientation of $L$.
\end{definition}

\begin{remark}
Given an oriented $3$-manifold $M$, let $\widehat{\mathcal{L}}_t(M)$ denote the free $\mathbb{C}[t,t^{-1}]$-module generated by isotopy classes of framed oriented links. This is distinct from $\mathcal{L}_t(M)$ in that we do not quotient out by the skein relations of Figure \ref{fig_skeinrelations}. The definition of $\Omega$ extends to give rise to a linear endomorphism of $\widehat{\mathcal{L}}_t(M)$ (but not of $\mathcal{L}_t(M)$).  As before, the identification \eqref{eqn_surfaceglue} makes $\widehat{\mathcal{L}}_t(\RS\times [0,1])$ into an algebra. In this case, the endomorphism defined by $\Omega$ is multiplicative.
\end{remark}

The following theorem, which explains the relationship between the skein $\Omega$ and the Hermite-Jacobi action, is Theorem 5.3 of \cite{gelcauribe1}.

\begin{theorem} \label{thm_skeinfourier}
Let $h_L\in\mcg{\RS}$ be a diffeomorphism that is represented by surgery on a framed link $L$; then the skein associated to the discrete Fourier transform $\rho(h_{L})$ by \eqref{eqn_linkthetarep} is $\Omega(L)$. Consequently, if we consider $\rho(h_{L})$ as an endomorphism of $\widetilde{\mathcal{L}}_t(H_g)$ using Theorem \ref{thm_thetaskein}, then
\begin{equation} \label{eqn_skeinfourier}
\rho(h_{L})[\beta] = \Omega(L)\cdot\beta, \quad \beta\in\widetilde{\mathcal{L}}_t(H_g).
\end{equation}
\end{theorem}
\noproof

\begin{remark}
Of course, since $\rho(h_{L})$ is a projective unitary representation, what is meant by \eqref{eqn_skeinfourier} is that left multiplication by the skein $\Omega(L)$ is a unitary representative for the equivalence class represented by $\rho(h_{L})$.
\end{remark}

\begin{remark} \label{rem_dehnaction}
Since the mapping class group is generated by Dehn twists, Theorem \ref{thm_skeinfourier} is sufficient to describe the Hermite-Jacobi action. Consider the subalgebra $E$ of $\widehat{\mathcal{L}}_t(\RS\times [0,1])$ that is linearly generated by isotopy classes of links for which surgery along that link does not change the diffeomorphism type of $\RS\times [0,1]$. We may consider the diffeomorphism represented by surgery on a framed link as a multiplicative map
\begin{displaymath}
\begin{array}{ccc}
E & \to & \mathbb{C}[\mcg{\RS}], \\
L & \mapsto & h_L.
\end{array}
\end{displaymath}

If
\[ h=h_{T_1^{\pm}}\circ h_{T_2^{\pm}}\circ\cdots\circ h_{T_n^{\pm}} \]
is a product of Dehn twists, it follows that $h$ is represented by surgery on the framed link $T_1^{\pm}\cdots T_n^{\pm}$. Consequently,
\[ \rho(h) = \rho(h_{T_1^{\pm}\cdots T_n^{\pm}}) = \Omega(T_1^{\pm}\cdots T_n^{\pm}). \]
\end{remark}

\section{The quantum group of abelian Chern-Simons theory} \label{sec_quantumgroup}

Any isotopy of knots can be decomposed into a sequence of  Reidemeister moves. Of them, the third Reidemeister move, an instance of which is depicted in Figure \ref{fig_reidemeister3}, was interpreted by Drinfeld \cite{drinfeld} as a symmetry which leads to the existence of quantum groups. It follows that the linking number skein modules, and hence the theory of classical theta functions, should have an associated quantum group. This is the quantum group of abelian Chern-Simons theory. In what follows we will explain how this quantum group is constructed and how the theory of classical theta functions is modeled using it.

\begin{figure}[ht]
\centering
\scalebox{.50}{\includegraphics{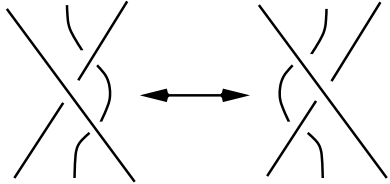}}
\caption{The third Reidemeister move leads to the study of quantum groups.}
\label{fig_reidemeister3}
\end{figure}

\subsection{The quantum group}

From now on we set $t=e^{\frac{i\pi}{N}}$, where $N$ is the even integer from Section \ref{sec_thetafun}. Note that $t$ is a primitive $2N$th root of unity. It is worth mentioning that, for the purposes of Sections \ref{sec_rmatrix} and \ref{sec_twist} alone, everything still holds true if $N$ is an odd integer. The quantum group associated to $U(1)$ is very simple; it is nothing more than the group algebra of the cyclic group $\mathbb{Z}_{2N}$ of order $2N$. However, in order to explain how one arrives at this as the correct definition of the quantum group for $U(1)$, we need to revisit the construction of the quantum group for $SL_2(\mathbb{C})$, from which it is deduced. The reader unconcerned with the derivation of this quantum group may skip ahead to Definition \ref{def_quantumgroup}.

Recall the Drinfeld-Jimbo construction (see e.g. \cite[\S XVII.2]{kassel})  of $\mathcal{U}_h(\mathfrak{sl}_2(\mathbb{C}))$. This is the quotient of the algebra $\mathbb{C}\langle X,Y,H \rangle[[h]]$ by the closure of the ideal generated by the relations
\[ [H,X]=2X, \quad [H,Y]=-2Y, \quad [X,Y]=\frac{e^{\frac{hH}{2}}-e^{-\frac{hH}{2}}}{e^{\frac{h}{2}}-e^{-\frac{h}{2}}} \]
in the $h$-adic topology. It carries the structure of a Hopf algebra;
\begin{displaymath}
\begin{array}{ll}
\Delta(H):=H\otimes 1 + 1\otimes H, & \Delta(X):=X\otimes e^{\frac{hH}{4}} + e^{-\frac{hH}{4}}\otimes X,\\
  \Delta(Y):=Y\otimes e^{\frac{hH}{4}} + e^{-\frac{hH}{4}}\otimes Y, &
S(H):=-H, \\ S(X):=-X, & S(Y):=-Y.
\end{array}
\end{displaymath}

There is a surjective map of Hopf algebras $\mathcal{U}_h(\mathfrak{sl}_2(\mathbb{C})) \to \mathcal{U}(\mathfrak{sl}_2(\mathbb{C}))$ defined by its action on the generators as follows;
\[ h\mapsto 0, \quad X\mapsto \left[\begin{array}{rr} 0 & 1 \\ 0 & 0\end{array}\right], \quad Y\mapsto \left[\begin{array}{rr} 0 & 0 \\ 1 & 0\end{array}\right], \quad H\mapsto \left[\begin{array}{rr} 1 & 0 \\ 0 & -1\end{array}\right]. \]

Now recall the definition, according to \cite[\S 8]{reshetikhinturaev}, of $\mathcal{U}_q(\mathfrak{sl}_2(\mathbb{C}))$. This is defined as the quotient of $\mathbb{C}\langle K,K^{-1},X,Y \rangle[q,q^{-1}]$ by the relations
\begin{displaymath}
\begin{array}{ll}
KK^{-1} = 1 = K^{-1}K & (q-q^{-1})[X,Y] = K^2-K^{-2} \\
KXK^{-1}=qX & KYK^{-1}=q^{-1}Y
\end{array}
\end{displaymath}
It carries the Hopf algebra structure
\begin{displaymath}
\begin{array}{lll}
\Delta(K)=K\otimes K, & \Delta(K^{-1})=K^{-1}\otimes K^{-1}, & S(K)=K^{-1}\\
 S(K^{-1}) = K, &
\Delta(X)=X\otimes K + K^{-1}\otimes X, & \Delta(Y)=Y\otimes K + K^{-1}\otimes Y,\\
 S(X)=-X, & S(Y)=-Y. &
\end{array}
\end{displaymath}

There is a map of Hopf algebras $\mathcal{U}_q(\mathfrak{sl}_2(\mathbb{C})) \to \mathcal{U}_h(\mathfrak{sl}_2(\mathbb{C}))$ defined by its action on the generators as follows;
\[ K \mapsto e^{\frac{hH}{4}}, \quad K^{-1} \mapsto e^{-\frac{hH}{4}}, \quad X \mapsto X, \quad Y \mapsto Y, \quad q \mapsto e^{\frac{h}{2}}. \]

Finally, Reshetikhin and Turaev \cite{reshetikhinturaev} define their quantum group $\mathcal{U}_t(\mathfrak{sl}_2(\mathbb{C}))$ as the quotient of $\mathcal{U}_q(\mathfrak{sl}_2(\mathbb{C}))$ by the relations
\[ K^{2N}=1, \quad X^{\frac{N}{2}}=Y^{\frac{N}{2}}=0, \quad q=t^2. \]

Having recalled the construction of the quantum group for $SL_2(\mathbb{C})$, we may now deduce from it the construction of the quantum group for $U(1)$. There is an inclusion of groups
\begin{displaymath}
\begin{array}{ccc}
U(1) & \to & SL_2(\mathbb{C}) \\
z & \mapsto & \left[\begin{array}{rr} z & 0 \\ 0 & \bar{z} \end{array}\right]
\end{array}
\end{displaymath}
giving rise to an inclusion of (real) Lie algebras
\begin{equation} \label{eqn_lieinc}
\begin{array}{rcl}
\mathfrak{u}(1)=\mathbb{R} & \to & \mathfrak{sl}_2(\mathbb{C}) \\
x & \mapsto & \left[\begin{array}{rr} ix & 0 \\ 0 & -ix \end{array}\right]
\end{array}
\end{equation}
If we denote by
\[ \mathfrak{u}_{\mathbb{C}}(1):=\mathbb{C}\otimes_{\mathbb{R}}\mathfrak{u}(1)=\mathbb{C} \]
the complexification of $\mathfrak{u}(1)$, then \eqref{eqn_lieinc} extends to a $\mathbb{C}$-linear map
\[ \mathfrak{u}_{\mathbb{C}}(1)\to\mathfrak{sl}_2(\mathbb{C}) \]
of complex Lie algebras.

We wish to define analogues $\mathcal{U}_h(\mathfrak{u}_{\mathbb{C}}(1))$, $\mathcal{U}_q(\mathfrak{u}_{\mathbb{C}}(1))$ and $\mathcal{U}_t(\mathfrak{u}_{\mathbb{C}}(1))$ of the above quantum groups for $SL_2(\mathbb{C})$ in such a way that they fit naturally into a commutative diagram
\begin{equation} \label{eqn_qgdiagram}
\xymatrix{ \mathcal{U}(\mathfrak{u}_{\mathbb{C}}(1)) \ar[r] & \mathcal{U}(\mathfrak{sl}_2(\mathbb{C})) \\ \mathcal{U}_h(\mathfrak{u}_{\mathbb{C}}(1)) \ar@{-->}[r] \ar@{-->}[u] & \mathcal{U}_h(\mathfrak{sl}_2(\mathbb{C})) \ar[u] \\ \mathcal{U}_q(\mathfrak{u}_{\mathbb{C}}(1)) \ar@{-->}[r] \ar@{-->}[u] \ar@{-->}[d] & \mathcal{U}_q(\mathfrak{sl}_2(\mathbb{C})) \ar[u] \ar[d] \\ \mathcal{U}_t(\mathfrak{u}_{\mathbb{C}}(1)) \ar@{-->}[r] & \mathcal{U}_t(\mathfrak{sl}_2(\mathbb{C})) }
\end{equation}

Since the Lie algebra $\mathfrak{u}_{\mathbb{C}}(1)$ is not semi-simple, we cannot use the Drinfeld-Jimbo construction to define its quantized enveloping algebra $\mathcal{U}_h(\mathfrak{u}_{\mathbb{C}}(1))$. However, by the definition of quantum enveloping algebra, we must have
\[ \mathcal{U}_h(\mathfrak{u}_{\mathbb{C}}(1)) = \mathcal{U}(\mathfrak{u}_{\mathbb{C}}(1))[[h]] \]
as a $\mathbb{C}[[h]]$-module. The only reasonable way to place a Hopf algebra structure on $\mathcal{U}_h(\mathfrak{u}_{\mathbb{C}}(1))$ so that it fits into the commutative diagram \eqref{eqn_qgdiagram} appears to be to choose the trivial deformation of $\mathcal{U}(\mathfrak{u}_{\mathbb{C}}(1))$. Consequently, since $\mathcal{U}(\mathfrak{u}_{\mathbb{C}}(1))$ is a polynomial algebra in a single variable, we define
\[ \mathcal{U}_h(\mathfrak{u}_{\mathbb{C}}(1)):=\mathbb{C}\langle H \rangle[[h]] \]
and define
\begin{displaymath}
\begin{array}{ll}
\mathcal{U}_h(\mathfrak{u}_{\mathbb{C}}(1)) \to \mathcal{U}_h(\mathfrak{sl}_2(\mathbb{C})); & h\mapsto h, \quad H\mapsto H \\
\mathcal{U}_h(\mathfrak{u}_{\mathbb{C}}(1)) \to \mathcal{U}(\mathfrak{u}_{\mathbb{C}}(1)); & h\mapsto 0, \quad H\mapsto [-i]\in\mathfrak{u}_{\mathbb{C}}(1)
\end{array}
\end{displaymath}

Similarly, one now realizes that the only reasonable way to define $\mathcal{U}_q(\mathfrak{u}_{\mathbb{C}}(1))$ is to define it as the quotient
\[ \mathcal{U}_q(\mathfrak{u}_{\mathbb{C}}(1)):=\mathbb{C}\langle K,K^{-1} \rangle[q,q^{-1}]/(KK^{-1}=1=K^{-1}K) \]
and choose the maps to be
\begin{displaymath}
\begin{array}{ll}
\mathcal{U}_q(\mathfrak{u}_{\mathbb{C}}(1)) \to \mathcal{U}_q(\mathfrak{sl}_2(\mathbb{C})); & K\mapsto K, \quad q\mapsto q \\
\mathcal{U}_q(\mathfrak{u}_{\mathbb{C}}(1)) \to \mathcal{U}_h(\mathfrak{u}_{\mathbb{C}}(1)); & K\mapsto e^{\frac{hH}{4}}, \quad q\mapsto e^{\frac{h}{2}}
\end{array}
\end{displaymath}

Finally, one arrives at the only sensible choice for $\mathcal{U}_t(\mathfrak{u}_{\mathbb{C}}(1))$.
\begin{definition} \label{def_quantumgroup}
The quantum group $\mathcal{U}_t(\mathfrak{u}_{\mathbb{C}}(1))$ is defined to be the quotient of the quantum group $\mathcal{U}_q(\mathfrak{u}_{\mathbb{C}}(1))$ by the ideal generated by the relations
\[ K^{2N}=1,\quad q=t^2. \]
\end{definition}
The quantum group $\mathcal{U}_t(\mathfrak{u}_{\mathbb{C}}(1))$ may be identified with the group algebra of $\mathbb{Z}_{2N}$ by identifying $K$ with the generator of $\mathbb{Z}_{2N}$. The remaining two maps at the bottom of \eqref{eqn_qgdiagram} are defined in the obvious manner.

The irreducible representations of $\mathcal{U}_t(\mathfrak{u}_{\mathbb{C}}(1))=\mathbb{C}[\mathbb{Z}_{2N}]$ are
\[ V^k; \quad k=0,1,\ldots, 2N-1 \]
where $V^k\cong \mathbb{C}$ and $K$ acts by $K\cdot v=t^{k}v$. We denote by $e_k$ the canonical basis element of $V^k$.

Because of the Hopf algebra structure, finite-dimensional representations form a ring (with the underlying abelian group defined \'a la Grothendieck from the monoid whose addition is the direct sum of representations) in which the product is provided by the tensor product of representations. Since $(V^1)^{\otimes k} \cong V^k$, the representation ring is
\begin{eqnarray*}
\mathbb{C}[V^1]/\left((V^1)^{2N}-1\right) = \mathbb{C}[\mathbb{Z}_{2N}].
\end{eqnarray*}
The fact that, in this case, the representation ring coincides with the quantum group, is purely coincidental. Since $\mathcal{U}_t(\mathfrak{u}_{\mathbb{C}}(1))$ is a Hopf algebra, this implies that the dual space of each representation is itself a representation. It is easy to see that there are  natural isomorphisms
\begin{eqnarray*}
D:(V^k)^*\rightarrow V^{2N-k},\quad De^k=e_k, \quad k=1,\ldots, 2N-1,
\end{eqnarray*}
where the functional $e^k$ is defined by $e^k(e_k)=1$.

In what follows, we will explain how $\mathcal{U}_t(\mathfrak{u}_{\mathbb{C}}(1))$ may be given the structure of a ribbon Hopf algebra. Everything will be phrased using the terminology from \cite{turaev}. Most of our computations are based on the simple fact that if $z\neq 1$ is a $k$th root of unity, then
\begin{equation} \label{eqn_rootsum}
1+z+z^2+\cdots+z^{k-1}=0.
\end{equation}

\subsection{The universal $R$-matrix} \label{sec_rmatrix}

Just as in the Reshetikhin-Turaev theory, we want to model the braiding of strands in Section \ref{sec_thetafun} by $R$-matrices. The $R$-matrix\footnote{This is actually $P\circ R$, where $P:V^m\otimes V^n\to V^n\otimes V^m$ transposes the factors.}  $V^m\otimes V^n\rightarrow V^n\otimes V^m$ should then come from the crossing of $m$ strands by $n$ strands, as in Figure \ref{fig_cross}, and hence should equal multiplication by $t^{mn}$ for all $m,n \in \{0,1,\ldots, 2N-1\}$.
\begin{figure}[ht]
\centering
\includegraphics{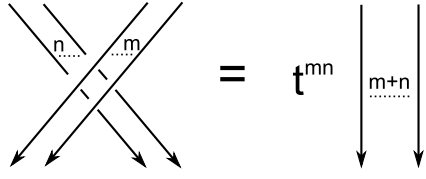}
\caption{The skein relation for the crossing of $n$ strands by $m$ strands.}
\label{fig_cross}
\end{figure}

We claim that, as for the quantum group of $SU(2)$ \cite{reshetikhinturaev}, these $R$-matrices are induced by a universal $R$-matrix in $\mathcal{U}_t(\mathfrak{u}_{\mathbb{C}}(1))$. The universal $R$-matrix should be of the form
\[ R:=\sum_{j,k=0}^{2N-1}c_{jk}K^j\otimes K^k. \]
Let us compute the coefficients $c_{jk}$.

Because $K^j=t^{jn}\id$ on $V^n$ for all $j$ and $n$, we obtain the system of equations
\begin{eqnarray} \label{eqn_rmatrixidentity}
\sum_{j,k=0}^{2N-1}c_{jk}t^{mj}t^{nk}=t^{mn}; \quad m,n\in \{0,1,\ldots, 2N-1\}.
\end{eqnarray}
If $T$ is the matrix whose $mn$th entry is $t^{mn}$ and $C=(c_{jk})$, then this equation becomes
\[ TCT=T \]
and hence we find that $C=T^{-1}$. Since $t$ is a primitive $2N$th root of unity, it follows from equation \eqref{eqn_rootsum} that
\begin{equation} \label{eqn_matinv}
c_{jk}=\frac{1}{2N}t^{-jk}.
\end{equation}
Hence, we arrive at the following formula for the $R$-matrix,
\begin{equation} \label{eqn_rmatrix}
R=\frac{1}{2N}\sum_{j,k\in\mathbb{Z}_{2N}} t^{-jk}K^j\otimes K^k.
\end{equation}
Note that this formula for $R$ implies that the $R$-matrix is symmetric in the sense that $P(R)=R$.

\begin{theorem} \label{thm_quasitriangular}
$(\mathcal{U}_t(\mathfrak{u}_{\mathbb{C}}(1)),R)$ is a quasi-triangular Hopf algebra.
\end{theorem}

\begin{proof}
We must show the following:
\begin{enumerate}
\item \label{item_quasitriangular1}
$R$ is invertible.
\item \label{item_quasitriangular2}
For all $a\in A$, $\Delta_{\mathrm{op}}(a)=R\Delta(a)R^{-1}$; where $\Delta_{\mathrm{op}}:=P\circ\Delta$.
\item \label{item_quasitriangular3}
The identities,
\begin{enumerate}
\item \label{item_quasitriangular3a}
$ R_{13}R_{12} = (\id\otimes\Delta)[R],$
\item \label{item_quasitriangular3b}
$R_{13}R_{23} = (\Delta\otimes\id)[R]. $
\end{enumerate}
\end{enumerate}

If $\mathcal{U}_t(\mathfrak{u}_{\mathbb{C}}(1))$ is to be a quasi-triangular Hopf algebra, then a formula for $R^{-1}$ should be given by
\[ R^{-1} = (S\otimes\id)[R] = \frac{1}{2N}\sum_{j,k\in\mathbb{Z}_{2N}} t^{-jk}K^{-j}\otimes K^k = \frac{1}{2N}\sum_{i,j\in\mathbb{Z}_{2N}} t^{jk}K^{-j}\otimes K^{-k}. \]
We may check that this element is inverse to $R$ as follows;
\begin{alignat*}{1}
RR^{-1} &= \frac{1}{4N^2}\sum_{j,j',k,k'\in\mathbb{Z}_{2N}}t^{j'k'-jk}K^{j-j'}
\otimes K^{k-k'} \\
&= \frac{1}{4N^2}\sum_{m,n\in\mathbb{Z}_{2N}}\left(\sum_{j,k\in\mathbb{Z}_{2N}}t^{(j-m)(k-n)-jk}\right)K^m\otimes K^n.
\end{alignat*}
But
\[ \sum_{j,k\in\mathbb{Z}_{2N}}t^{(j-m)(k-n)-jk} = t^{mn}\left(\sum_{j\in\mathbb{Z}_{2N}}t^{-jn}\right)\left(\sum_{k\in\mathbb{Z}_{2N}}t^{-km}\right). \]
By \eqref{eqn_rootsum}, this  is zero unless $m=n=0$, in which case it is equal to $4N^2$. Hence $RR^{-1}=1\otimes 1$.

Having proven \eqref{item_quasitriangular1}, item \eqref{item_quasitriangular2} follows trivially from the fact that $\mathcal{U}_t(\mathfrak{u}_{\mathbb{C}}(1))=\mathbb{C}[\mathbb{Z}_{2N}]$ is both commutative and co-commutative.

Finally, to establish \eqref{item_quasitriangular3}, one may compute directly from equation \eqref{eqn_rmatrix} that
\begin{alignat*}{1}
R_{13}R_{12} &= \frac{1}{4N^2}\sum_{j,j',k,k'\in\mathbb{Z}_{2N}}t^{-jk-j'k'} K^{j+j'}\otimes K^{k'}\otimes K^k \\
&= \frac{1}{4N^2}\sum_{m\in\mathbb{Z}_{2N}}\sum_{j,k,k'\in\mathbb{Z}_{2N}}t^{-jk-(m-j)k'} K^m\otimes K^{k'}\otimes K^k \\
&= \frac{1}{4N^2}\sum_{m,k,k'\in\mathbb{Z}_{2N}}t^{-mk'}\left(\sum_{j\in\mathbb{Z}_{2N}}t^{j(k'-k)}\right) K^m\otimes K^{k'}\otimes K^k
\end{alignat*}
Now by \eqref{eqn_rootsum}, $\sum_{j\in\mathbb{Z}_{2N}}t^{j(k'-k)}$ is equal to zero, unless $k'=k$, in which case it is equal to $2N$. Hence,
\begin{alignat*}{1}
R_{13}R_{12} &= \frac{1}{2N}\sum_{m,k\in\mathbb{Z}_{2N}}t^{-mk} K^m\otimes K^k\otimes K^k =(\id\otimes\Delta)[R]
\end{alignat*}
which establishes \eqref{item_quasitriangular3a}. Item \eqref{item_quasitriangular3b} follows from a similar argument.
\end{proof}

\subsection{The universal twist} \label{sec_twist}

We wish to show that the quasi-triangular Hopf algebra $(\mathcal{U}_t(\mathfrak{u}_{\mathbb{C}}(1)),R)$ of the previous section may be extended to a ribbon Hopf algebra by constructing a universal twist $v\in Z(\mathcal{U}_t(\mathfrak{u}_{\mathbb{C}}(1)))=\mathcal{U}_t(\mathfrak{u}_{\mathbb{C}}(1))$. Multiplication on the left of $V^k$ by this element is intended to model a positive right-hand twist, as in Figure \ref{fig_twist}.
\begin{figure}[ht]
\centering
\includegraphics{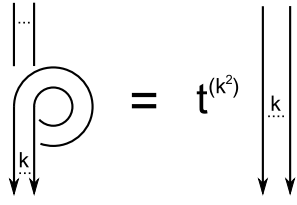}
\caption{The skein relation for a positive twist of $k$ strands.}
\label{fig_twist}
\end{figure}

If we write $v$ as
\[v=\sum_{j\in\mathbb{Z}_{2N}}\mu_j K^j\]
then this observation leads us to the equation
\begin{equation} \label{eqn_twistidentity}
\sum_{j\in\mathbb{Z}_{2N}}\mu_j t^{jk}=t^{k^2}
\end{equation}
for the coefficients $\mu_j$. From the formula \eqref{eqn_matinv} for $T^{-1}$ derived in Section \ref{sec_rmatrix}, we find
\[ \mu_j = \frac{1}{2N}\sum_{l\in\mathbb{Z}_{2N}}t^{-jl}t^{l^2}, \]
which yields
\begin{eqnarray*}
v = \frac{1}{2N}\sum_{j,k\in\mathbb{Z}_{2N}}t^{k(k-j)}K^j
= \frac{1}{2N}\sum_{j,k\in\mathbb{Z}_{2N}}t^{(k+j)k}K^j.
\end{eqnarray*}
Before proceeding further, we determine a simpler expression for  $v$. Note that
\begin{alignat*}{1}
\sum_{k\in\mathbb{Z}_{2N}}t^{(k+j)k} &= \sum_{k=0}^{N-1}\left(t^{(k+j)k}+t^{(k+N+j)(k+N)}\right) \\&= \sum_{k=0}^{N-1}t^{(k+j)k}(1+t^{2kN+jN+N^2})
= (1+(-1)^{j+N})\sum_{k=0}^{N-1}t^{(k+j)k}.
\end{alignat*}
Hence
\begin{alignat*}{1}
v &= \frac{1}{2N}\sum_{j\in\mathbb{Z}_{2N}}\left[(1+(-1)^{j+N})\sum_{k=0}^{N-1}t^{(k+j)k}K^j\right] \\
&= \frac{1}{N}\sum_{j=0}^{N-1}\sum_{k=0}^{N-1}t^{(k+2j+N)k}K^{2j+N}
= \frac{1}{N}\sum_{j,k\in\mathbb{Z}_{N}}(-1)^k t^{(k+2j)k}K^{2j+N} \\
&= \frac{1}{N}\sum_{j,k\in\mathbb{Z}_{N}}(-1)^{k-j} t^{(k+j)(k-j)}K^{2j+N} \\
\end{alignat*}
yielding the formula
\begin{equation} \label{eqn_twist}
v = \frac{1}{N}\left(\sum_{k\in\mathbb{Z}_{N}}(-1)^k t^{k^2}\right)\left(\sum_{j\in\mathbb{Z}_{N}}(-1)^{j} t^{-j^2}K^{2j+N}\right).
\end{equation}

\begin{theorem} \label{thm_twist}
$(\mathcal{U}_t(\mathfrak{u}_{\mathbb{C}}(1)),R,v)$ is a ribbon Hopf algebra.
\end{theorem}

\begin{proof}
We must show the following:
\begin{enumerate}
\item \label{item_twist1}
$v$ is invertible,
\item \label{item_twist2}
$ S(v)=v,$
\item \label{item_twist3}
$ \Delta(v)=P(R)R(v\otimes v). $
\end{enumerate}

To prove \eqref{item_twist1}, consider that the inverse of $v$ should model a left-hand twist. This suggests looking for a formula for $v^{-1}$ by replacing $t$ with $t^{-1}$ in \eqref{eqn_twist}. This yields the putative formula
\[ v^{-1} = \frac{1}{N}\left(\sum_{k\in\mathbb{Z}_{N}}(-1)^k t^{-k^2}\right)\left(\sum_{j\in\mathbb{Z}_{N}}(-1)^{j} t^{j^2}K^{2j+N}\right). \]

One checks that this element is inverse to $v$ as follows. Firstly
\begin{eqnarray}\label{eqn_squareprod}
&&\left(\sum_{k\in\mathbb{Z}_N}(-1)^k t^{k^2} \right) \left(\sum_{k'\in\mathbb{Z}_N}(-1)^{k'} t^{-k'^2} \right) = \sum_{k,k'\in\mathbb{Z}_N}(-1)^{k+k'}t^{(k-k')(k+k')}
\notag\\
&&\quad =\sum_{k,k'\in\mathbb{Z}_N}(-1)^k t^{k(k+2k')}
= \sum_{k\in\mathbb{Z}_N}\left[(-1)^k t^{k^2}\sum_{k'\in\mathbb{Z}_N}t^{2kk'}\right]
= N
\end{eqnarray}
where the last equality follows from \eqref{eqn_rootsum}. Hence
\begin{alignat*}{1}
vv^{-1} &= \frac{1}{N}\left(\sum_{j\in\mathbb{Z}_{N}}(-1)^{j} t^{-j^2}K^{2j+N}\right)\left(\sum_{j'\in\mathbb{Z}_{N}}(-1)^{j'} t^{j'^2}K^{2j'+N}\right) \\
&= \frac{1}{N}\sum_{j,j'\in\mathbb{Z}_{N}}(-1)^{j+j'} t^{j'^2-j^2}K^{2j+2j'}
= \frac{1}{N}\sum_{m\in\mathbb{Z}_N}\sum_{j\in\mathbb{Z}_{N}}(-1)^{m} t^{(m-j)^2-j^2}K^{2m} \\
&= \frac{1}{N}\sum_{m\in\mathbb{Z}_N}(-1)^{m}t^{m^2}\left(\sum_{j\in\mathbb{Z}_{N}} t^{-2mj}\right)K^{2m} = 1,
\end{alignat*}
where the last line follows from \eqref{eqn_rootsum}, proving that $v$ is invertible as claimed.

To show \eqref{item_twist2}, note that $S(K^{2j+N})=K^{-2j-N}=K^{-2j+N}$, hence
\begin{alignat*}{1}
S(v) &= \frac{1}{N}\left(\sum_{k\in\mathbb{Z}_{N}}(-1)^k t^{k^2}\right)\left(\sum_{j\in\mathbb{Z}_{N}}(-1)^{j} t^{-j^2}K^{-2j+N}\right) \\
&= \frac{1}{N}\left(\sum_{k\in\mathbb{Z}_{N}}(-1)^k t^{k^2}\right)\left(\sum_{j\in\mathbb{Z}_{N}}(-1)^{-j} t^{-(-j)^2}K^{2j+N}\right) = v.
\end{alignat*}

Finally, since $P(R)=R$, to show \eqref{item_twist3} we need to compute $R^2$. From \eqref{eqn_rmatrix} we calculate
\begin{alignat*}{1}
R^2 &= \frac{1}{4N^2}\sum_{j,j',k,k'\in\mathbb{Z}_{2N}}t^{-jk-j'k'}K^{j+j'}\otimes K^{k+k'} \\
&= \frac{1}{4N^2} \sum_{m,n\in\mathbb{Z}_{2N}}\sum_{j,k\in\mathbb{Z}_{2N}} t^{-jk-(m-j)(n-k)}K^m\otimes K^n \\
&= \frac{1}{4N^2} \sum_{m,n,k\in\mathbb{Z}_{2N}}t^{m(k-n)}\left(\sum_{j\in\mathbb{Z}_{2N}} t^{(n-2k)j}\right)K^m\otimes K^n \\
&= \frac{1}{2N} \sum_{m,k\in\mathbb{Z}_{2N}}t^{-mk}K^m\otimes K^{2k} \\
&= \frac{1}{2N} \sum_{m\in\mathbb{Z}_{2N}}\sum_{k=0}^{N-1}(t^{-mk}+t^{-m(k+N)})K^m\otimes K^{2k} \\
&= \frac{1}{2N} \sum_{m\in\mathbb{Z}_{2N}}\left[(1+(-1)^m)\sum_{k=0}^{N-1}t^{-mk}K^m\otimes K^{2k}\right] \\
&= \frac{1}{N} \sum_{m=0}^{N-1}\sum_{k=0}^{N-1}t^{-2mk}K^{2m}\otimes K^{2k}.
\end{alignat*}
where line 4 follows from \eqref{eqn_rootsum}. This  yields the following formula for $R^2$,
\begin{equation} \label{eqn_rmatrixsquared}
R^2 = \frac{1}{N} \sum_{j,k\in\mathbb{Z}_N}t^{-2jk}K^{2j}\otimes K^{2k}.
\end{equation}

Next, using \eqref{eqn_twist}, we may write
\[ v\otimes v = \frac{1}{N^2}\left(\sum_{r\in\mathbb{Z}_{N}}(-1)^r t^{r^2}\right)^2\left(\sum_{j,k\in\mathbb{Z}_{N}}(-1)^{j+k} t^{-j^2-k^2}K^{2j+N}\otimes K^{2k+N}\right). \]
From this and \eqref{eqn_rmatrixsquared}, we obtain that $R^2(v\otimes v)$ is equal to
\begin{alignat*}{1}
& \frac{1}{N^3}\left(\sum_{r\in\mathbb{Z}_{N}}(-1)^r t^{r^2}\right)^2\left(\sum_{m,n,j,k\in\mathbb{Z}_N}(-1)^{j+k} t^{-2mn-j^2-k^2}K^{2m+2j+N}\otimes K^{2n+2k+N}\right) \\
&= \frac{1}{N^3}\left(\sum_{r\in\mathbb{Z}_{N}}(-1)^r t^{r^2}\right)^2\left(\sum_{s,s'\in\mathbb{Z}_N}\left(\sum_{j,k\in\mathbb{Z}_N}(-1)^{j+k} t^{-2(s-j)(s'-k)-j^2-k^2}\right)K^{2s+N}\otimes K^{2s'+N}\right) \\
&= \frac{1}{N^3}\left(\sum_{r\in\mathbb{Z}_{N}}(-1)^r t^{r^2}\right)^2\left(\sum_{s,s'\in\mathbb{Z}_N}t^{-2ss'}\left(\sum_{j,k\in\mathbb{Z}_N}(-1)^{j+k} t^{-(j+k)^2+2sk+2s'j}\right)K^{2s+N}\otimes K^{2s'+N}\right).
\end{alignat*}
Now  consider the coefficient
\begin{alignat*}{1}
\sum_{j,k\in\mathbb{Z}_N}(-1)^{j+k} t^{-(j+k)^2+2sk+2s'j} &= \sum_{j,k\in\mathbb{Z}_N}(-1)^j t^{-j^2+2sk+2s'(j-k)} \\
&= \left(\sum_{j\in\mathbb{Z}_N}(-1)^j t^{-j^2+2s'j}\right)\left(\sum_{k\in\mathbb{Z}_N}t^{2(s-s')k}\right).
\end{alignat*}
By \eqref{eqn_rootsum}, the right-hand factor is zero, unless $s=s'$, in which case it is equal to $N$. The left-hand factor is
\begin{eqnarray*}
\sum_{j\in\mathbb{Z}_N}(-1)^j t^{-j^2+2s'j} = t^{s'^2}\sum_{j\in\mathbb{Z}_N}(-1)^j t^{-(j-s')^2} = (-1)^{s'}t^{s'^2}\sum_{j\in\mathbb{Z}_N}(-1)^j t^{-j^2}.
\end{eqnarray*}
Consequently, $R^2(v\otimes v)$ equals
\begin{alignat*}{1}
& \frac{1}{N^2}\left(\sum_{r\in\mathbb{Z}_{N}}(-1)^r t^{r^2}\right)^2\left(\sum_{j\in\mathbb{Z}_N}(-1)^j t^{-j^2}\right)\left(\sum_{s\in\mathbb{Z}_N}(-1)^s t^{-s^2}K^{2s+N}\otimes K^{2s+N}\right) \\
&= \frac{1}{N}\left(\sum_{r\in\mathbb{Z}_{N}}(-1)^r t^{r^2}\right)\left(\sum_{s\in\mathbb{Z}_N}(-1)^s t^{-s^2}K^{2s+N}\otimes K^{2s+N}\right)=\Delta(v),
\end{alignat*}
where on the last line we have used \eqref{eqn_squareprod}. This proves \eqref{item_twist3}.
\end{proof}

\section{Modeling classical theta functions using the quantum group}

\subsection{Quantum link invariants}
Theorem \ref{thm_twist} implies that the quantum group $A:=\mathcal{U}_t(\mathfrak{u}_{\mathbb{C}}(1))$  can be used to define invariants of  knots and links in $S^3$, the 3-dimensional sphere. Let us recall the construction, which is described in detail in the general setting in \cite{reshetikhinturaev} and \cite{turaev}.

View $S^3$ as ${\mathbb R}^3$ compactified with the point at infinity, and in it fix a plane  and a direction in the plane called the vertical direction. Given an oriented framed link $L$ in $S^3$, deform it through an isotopy to a link whose framing is parallel to the plane, and  whose projection onto the plane is a link diagram that can be sliced by finitely many horizontal lines into pieces, each of which consists of several vertical strands and exactly one of the events from Figure \ref{fig_shortphenomena}, maybe with reversed arrows.

\begin{figure}[ht]
\centering
\includegraphics{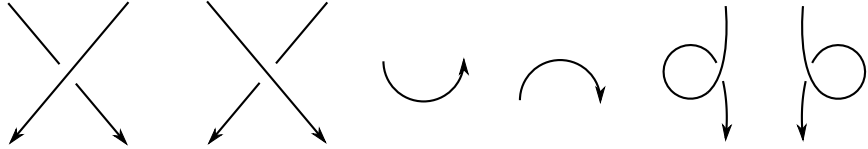}
\caption{Any link diagram may be decomposed into a sequence involving the above events.}
\label{fig_shortphenomena}
\end{figure}

Now consider a coloring $\mathbf{V}$ of the components of $L$ by irreducible representations of $\qgr$. This means that if $L$ has components $L_1,L_2,\ldots,L_m$,
then $\mathbf{V}$ is a map
\begin{eqnarray*}
\mathbf{V}:\{L_1,L_2,\ldots, L_m\}\rightarrow \{V^0,V^1,\ldots, V^{2N-1}\}.
\end{eqnarray*}
We shall denote the data consisting of $L$ together with its coloring $\mathbf{V}$ by $\mathbf{V}(L)$. Each of the horizontal lines considered is crossed by finitely many strands of $L$. To a strand that crosses downwards, associate to that crossing point the corresponding irreducible representation decorating that strand; to a strand that crosses upwards, associate the dual of that representation. Then take the tensor product of all these irreducible representations for that given horizontal line, from left to right. This defines a representation of $\qgr$. If the horizontal line does not intersect the link, the representation is automatically $V^0$.

The link diagram defines a composition of homomorphisms between the representations on each horizontal line when traveling in the vertical direction from below the link diagram to above the link diagram. The events from Figure \ref{fig_shortphenomena} define the following operators, from left to right
\begin{itemize}
\item
$\check{R}:=P\circ R$, where $R$ denotes multiplication by the $R$-matrix \eqref{eqn_rmatrix};
\item
$\check{R}^{-1}=R^{-1}\circ P$;
\item
the operator $V^0\to V^k\otimes (V^k)^*$ given by $e_0\mapsto e_{k}\otimes e^k$;
\item
the operator $(V^k)^*\otimes V^k\rightarrow V^0$ given by $e^k\otimes e_k\mapsto e_0$;
\item
the twist, given by multiplying by $v\in\qgr$;
\item
the inverse of the twist, given by multiplying by $v^{-1}$.
\end{itemize}
Vertical strands away from the events define the identity operators, which are then tensored with the operators of the events to form a homomorphism between the representations associated to the horizontal lines.

The colored link diagram defines an endomorphism of  $V^0=\mathbb{C}$, which is given by multiplication by a complex number $<\mathbf{V}(L)>$. Because $\qgr$ has the structure of a ribbon Hopf algebra, it follows from \cite[\S I.2, XI.3]{turaev} that $<\mathbf{V}(L)>$ is an invariant of colored framed oriented links, meaning that it depends only on the isotopy class of the link and not on its projection onto the plane.

In fact, according to \cite{turaev}, we may color these links using any representations of $\qgr$, and the above algorithm still leads to an isotopy invariant. There is even a simple ``cabling formula'', which states that if a link component is colored by the representation $U\otimes W$, then we may replace that link component by two parallel copies, one colored by $U$ and the other by $W$. In particular, if our link component is colored by $V^k$, then we may replace it with $k$ parallel copies which are each colored by $V^1$.

Furthermore, since this invariant is distributive with respect to direct sums of representations, its definition extends to links that are colored by elements of the representation ring of our quantum group $\qgr$. We may describe this extension explicitly as follows. If $L$ is a link whose components $L_1,\ldots, L_m$ are colored by elements of the representation ring of $\qgr$,
\begin{equation} \label{eqn_linkcoloring}
\mathbf{V}:L_j\mapsto\sum_{k=0}^{2N-1}c_{jk}V^k, \quad 1\leq j\leq m;
\end{equation}
then
\begin{eqnarray*}
<\mathbf{V}(L)>=\sum_{k_1=0}^{2N-1}\sum_{k_2=0}^{2N-1}\cdots
\sum_{k_m=0}^{2N-1}c_{1k_1}c_{2k_2}\cdots c_{mk_m}<\mathbf{V}_{k_1,k_2,\ldots,k_m}(L)>;
\end{eqnarray*}
where $\mathbf{V}_{k_1,k_2,\ldots, k_m}$ is the coloring of $L$ that decorates the component $L_j$ with the color $V^{k_j}$.

We wish to connect these colored links and their invariants to the skein theory of Section \ref{sec_thetafun}. This is the purpose of our next theorem, for which we will need the following definition.

\begin{definition} \label{def_coloredlinks}
Given an oriented 3-manifold $M$, we define $\mathcal{V}_{\qgr}(M)$ to be the vector space whose basis consists of isotopy classes of oriented framed links, whose components are colored by \emph{irreducible} representations of $\qgr$.

There is a map
\[ \mathcal{V}_{\qgr}(M) \to \widetilde{\mathcal{L}}_t(M) \]
which is defined by the cabling formula which replaces any link component colored by $V^j$ with $j$ parallel copies of that component.
\end{definition}

\begin{theorem} \label{thm_main}
The following diagram commutes
\begin{displaymath}
\xymatrix{ \mathcal{V}_{\qgr}(S^3) \ar[d]_{\mathrm{cabling}} \ar[rd]^{\mathrm{invariant}} \\ \widetilde{\mathcal{L}}_t(S^3) \ar@{=}[r] & \mathbb{C} }
\end{displaymath}
where the map on the right assigns to a colored link $\mathbf{V}(L)$, its invariant $<\mathbf{V}(L)>$, and the map on the left is given by the cabling formula of Definition \ref{def_coloredlinks}.
\end{theorem}

\begin{proof}
This is a basic consequence of our construction of the quantum group $\qgr:=\mathcal{U}_t(\mathfrak{u}_\mathbb{C}(1))$ which we carried out in Section \ref{sec_quantumgroup}. Present the colored link diagram as a composition of horizontal slices, each of which contains exactly one event from Figure \ref{fig_shortphenomena}. Using the canonical basis elements $e_j$ of $V^j$, the representation associated to each horizontal line may be identified with $\mathbb{C}$. Consequently, the homomorphism assigned to each horizontal slice is just multiplication by some complex number. For the events listed in Figure \ref{fig_shortphenomena}, these numbers are respectively
\[ t^{mn},t^{-mn},1,1,t^{k^2},t^{-k^2}; \]
where we have assumed that the strands of both the two left diagrams of Figure \ref{fig_shortphenomena} are labeled by $V^m$ and $V^n$ and that each strand of the two right diagrams is labeled by $V^k$. These numbers are a consequence of equations \eqref{eqn_rmatrixidentity} and \eqref{eqn_twistidentity}. The result is now a consequence of the identities detailed in figures \ref{fig_cross} and \ref{fig_twist}. Of course, the first two events of Figure \ref{fig_shortphenomena} may appear with a reversed arrow on the bottom right, but one may easily check that this does not affect the result.
\end{proof}

\begin{cor} \label{cor_colreplace}
Let $\mathbf{V}$ be a coloring of a link $L$ in $S^3$ by irreducible representations, and suppose that some link component of $L$ is colored by $V^n$ with $0\leq n\leq N-1$. If $\mathbf{V}'$ denotes the coloring of $L$ obtained
by replacing the color of that link component by $V^{n+N}$, then
\[ <\mathbf{V}(L)> = <\mathbf{V}'(L)>. \]
\end{cor}
\noproof

It follows that we may factor the representation ring by the ideal generated by the single relation $V^N=1$, without any affect on the invariants $<\mathbf{V}(L)>$. Let
\[ R(\qgr):=\mathbb{C}[V^1]/\left((V^1)^N-1\right) \]
denote this quotient. Note that in $R(\qgr)$,
\begin{eqnarray*}
(V^k)^*=V^{N-k} \quad\text{and}\quad V^{m}\otimes V^{n}=V^{m+n(\mathrm{mod}N)};
\end{eqnarray*}
where these equalities should be interpreted formally (i.e. inside $R(\qgr)$). Thus, we can think of the link invariant defined above as an invariant of links colored by elements of $R(\qgr)$.

\subsection{Theta functions as colored links in a handlebody}

Consider the Heegaard decomposition
\[ H_g \bigsqcup_{\partial H_g \approx \partial H_g} H_g = S^3 \]
of $S^3$ given by \eqref{eqn_heegaardmap}. This decomposition gives rise to a bilinear pairing
\begin{equation} \label{eqn_linkpairing}
\begin{array}{rccccc}
[ \cdot , \cdot ]: & \mathcal{V}_{\qgr}(H_g) \otimes \mathcal{V}_{\qgr}(H_g) & \to & \mathcal{V}_{\qgr}(S^3) & \to & \mathbb{C} \\
 & \mathbf{V}(L) \otimes \mathbf{V}'(L') & \mapsto & \mathbf{V}(L) \cup \mathbf{V}'(L') & \mapsto & <\mathbf{V}(L) \cup \mathbf{V}'(L')>
\end{array}
\end{equation}
on $\mathcal{V}_{\qgr}(H_g)$. Due to the obvious diffeomorphism of $S^3$ that swaps the component handlebodys, this pairing is symmetric. However, it is far from being nondegenerate, which leads us to the following definition.

\begin{definition}
The vector space $\widetilde{\mathcal{L}}_{\qgr}(H_g)$ is defined to be the quotient of the vector space $\mathcal{V}_{\qgr}(H_g)$ by the annihilator
\[ \Ann(\mathcal{V}_{\qgr}(H_g)):=\{ x\in\mathcal{V}_{\qgr}(H_g): [x,y]=0, \text{ for all } y\in\mathcal{V}_{\qgr}(H_g). \} \]
of the form \eqref{eqn_linkpairing}.
\end{definition}

The pairing induced on $\widetilde{\mathcal{L}}_{\qgr}(H_g)$ by \eqref{eqn_linkpairing} is nondegenerate.

\begin{proposition} \label{prop_thetaquantiso}
Consider the cabling map from $\mathcal{V}_{\qgr}(H_g)$ to $\widetilde{\mathcal{L}}_t(H_g)$ described in Definition \ref{def_coloredlinks}. This map factors to an isomorphism,
\[ \widetilde{\mathcal{L}}_{\qgr}(H_g) \cong \widetilde{\mathcal{L}}_t(H_g). \]
\end{proposition}

\begin{proof}
By Theorem \ref{thm_main}, the following diagram commutes;
\begin{displaymath}
\xymatrix{ \mathcal{V}_{\qgr}(H_g)\otimes\mathcal{V}_{\qgr}(H_g) \ar[rr]^{\mathrm{cabling}} \ar[d] && \widetilde{\mathcal{L}}_t(H_g)\otimes\widetilde{\mathcal{L}}_t(H_g) \ar[d] \\ \mathcal{V}_{\qgr}(S^3) \ar[rr]^{\mathrm{cabling}} \ar[rd]_{\mathrm{invariant}} && \widetilde{\mathcal{L}}_t(S^3) \ar@{=}[ld] \\ & \mathbb{C} }
\end{displaymath}
Since the cabling map from $\mathcal{V}_{\qgr}(H_g)$ to $\widetilde{\mathcal{L}}_t(H_g)$ is obviously surjective, the result follows from the fact that the pairing \eqref{eqn_pairingskein} appearing on the right of the above diagram is nondegenerate.
\end{proof}

\begin{cor} \label{cor_thetacolor}
The space $\widetilde{\mathcal{L}}_{\qgr}(H_g)$ is isomorphic to the space of theta functions $\spacetheta(\RS)$.
\end{cor}

\begin{proof}
This is a consequence of Theorem \ref{thm_thetaskein}.
\end{proof}

Consequently, we may represent the theta series $\theta^{\Pi}_{k}$ as colored links in the handlebody $H_g$. More precisely, the coloring represented by Figure \ref{fig_thetabasis} of $a_1,\ldots, a_g$ by $V^{k_1},\ldots, V^{k_g}$ respectively, corresponds to the theta series $\theta^{\Pi}_{k_1,\ldots, k_g}$.
\begin{figure}[ht]
\centering
\scalebox{.30}{\includegraphics{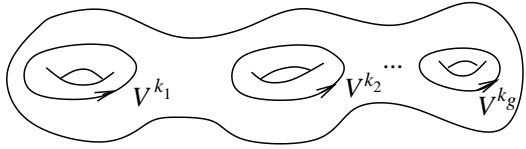}}
\caption{The presentation of $\theta^{\Pi}_{k_1,\ldots, k_g}$ as a colored link in a handlebody.}
\label{fig_thetabasis}
\end{figure}

\subsection{The Schr\"odinger representation and the Hermite-Jacobi action via quantum group representations}

Consider the spaces $\mathcal{V}_{\qgr}(\RS\times [0,1])$ and $\mathcal{V}_{\qgr}(H_g)$. As before, by \eqref{eqn_surfaceglue}, we see that $\mathcal{V}_{\qgr}(\RS\times [0,1])$ is an algebra and, by \eqref{eqn_bdryglue}, that $\mathcal{V}_{\qgr}(H_g)$ is a module over this algebra. Since the cabling maps to the corresponding reduced skein modules are equivariant with respect to this action, it follows from Proposition \ref{prop_thetaquantiso} that this action descends to $\widetilde{\mathcal{L}}_{\qgr}(H_g)$.

We wish to cast the action of the finite Heisenberg group on the space of theta functions within this setting. Given $p,q\in\{0,\ldots,N-1\}^g$, set
\[ \Gamma(p,q):= a_1^{V^{p_1}}a_2^{V^{p_2}}\cdots a_g^{V^{p_g}} b_1^{V^{q_1}}b_2^{V^{q_2}}\cdots b_g^{V^{q_g}}; \]
where $\gamma^{V^j}\in\mathcal{V}_{\qgr}(\RS\times [0,1])$ denotes the curve $\gamma$ colored by the irreducible representation $V^j$.

\begin{proposition}
Let $\Phi:\spacetheta(\RS)\to\widetilde{\mathcal{L}}_{\qgr}(H_g)$ denote the isomorphism of Corollary \ref{cor_thetacolor}, then
\[ \Phi[(p,q,k)\cdot\theta] = t^{(k-p^Tq)}\Gamma(p,q)\cdot\Phi[\theta]; \quad (p,q,k)\in\heis(\mathbb{Z}_N^g),\theta\in\spacetheta(\RS). \]
\end{proposition}

\begin{proof}
Consider the commutative diagram
\begin{displaymath}
\xymatrix{ \mathcal{V}_{\qgr}(\RS\times [0,1])\otimes\widetilde{\mathcal{L}}_{\qgr}(H_g) \ar[r]^-{\mathrm{cabling}} \ar[d] & \widetilde{\mathcal{L}}_t(\RS\times [0,1])\otimes\widetilde{\mathcal{L}}_t(H_g) \ar[d] & \mathbb{C}(\heis(\mathbb{Z}_N^g))\otimes\spacetheta(\RS) \ar[d] \ar[l] \\ \widetilde{\mathcal{L}}_{\qgr}(H_g) \ar[r]^-{\mathrm{cabling}}_-{\cong} & \widetilde{\mathcal{L}}_t(H_g) & \spacetheta(\RS) \ar[l]^-{\cong} }
\end{displaymath}
where we have used Proposition \ref{prop_thetaquantiso} and Theorems \ref{thm_linkingheisenberg} and \ref{thm_thetaskein}. Since the image of $(p,q,k)$ in $\widetilde{\mathcal{L}}_t(\RS\times [0,1])$ under \eqref{eqn_linkheis} coincides with the image of $t^{(k-p^Tq)}\Gamma(p,q)$ under the cabling map, this proves the result.
\end{proof}

Note that the following corollary is the abelian analogue of the main result in \cite{gelcauribe0}.

\begin{cor}
The Weyl quantization and the quantum group quantization of the moduli space of flat $\mathfrak{u}(1)$-connections on a closed surface are equivalent.
\end{cor}
\noproof

Denote by $\mathcal{V}'_{\qgr}(\RS\times [0,1])$ the vector space that is freely generated by isotopy classes of links colored by elements of the \emph{representation ring} of $\qgr$. There is a multiplicative map from this space onto $\mathcal{V}_{\qgr}(\RS\times [0,1])$; if $L$ is a link with $m$ components whose coloring $\mathbf{V}$ by the representation ring is written as in \eqref{eqn_linkcoloring}, then this map is given by
\[ \mathbf{V}(L) \mapsto \sum_{k_1,\ldots,k_m=0}^{2N-1}c_{1k_1}\cdots c_{mk_m}\mathbf{V}_{k_1,\ldots,k_m}(L); \]
where $\mathbf{V}_{k_1,\ldots, k_m}$ is the coloring of $L$ that decorates the component $L_j$ with the color $V^{k_j}$.

Consequently, $\widetilde{\mathcal{L}}_{\qgr}(H_g)$ is a module over $\mathcal{V}'_{\qgr}(\RS\times [0,1])$. In fact, in view of Corollary \ref{cor_colreplace}, we may assume that our links are colored by elements of the quotient ring $R(\qgr)$. If $L$ is an oriented framed link in $\RS\times [0,1]$, we denote by $\Omega_{\qgr}(L)$ the element of $\mathcal{V}'_{\qgr}(\RS\times [0,1])$ that is obtained by coloring each component of $L$ by $N^{-\frac{1}{2}}\sum_{k=0}^{N-1} V^k$.

\begin{proposition}
Let $h_L\in\mcg{\RS}$ be a diffeomorphism that is represented by surgery on a framed link $L$. By Corollary \ref{cor_thetacolor} we may consider the discrete Fourier transform $\rho(h_{L})$ as an endomorphism of $\widetilde{\mathcal{L}}_{\qgr}(H_g)$. This endomorphism is given (projectively) by:
\[ \rho(h_{L})[\beta] = \Omega_{\qgr}(L)\cdot\beta, \quad \beta\in\widetilde{\mathcal{L}}_{\qgr}(H_g). \]
\end{proposition}

\begin{proof}
Since the image of $\Omega_A(L)$ in $\widetilde{\mathcal{L}}_t(\RS\times [0,1])$ under the cabling map coincides with $\Omega(L)$, this is a consequence of Theorem \ref{thm_skeinfourier} and the fact that the cabling map is equivariant.
\end{proof}

\begin{remark}
Of course, by Remark \ref{rem_dehnaction}, it suffices to be able to apply the above proposition to Dehn twists.
\end{remark}

\end{document}